\documentclass[11pt]{amsart}
\usepackage{enumerate}
\usepackage{graphicx}
\usepackage{amssymb}
\usepackage{amsthm}

 \newtheorem{theorem}{Theorem}[section]
\newtheorem{lemma}[theorem]{Lemma}
\newtheorem{definition}[theorem]{Definition}

\newtheorem{proposition}[theorem]{Proposition}
\newtheorem{corollary}[theorem]{Corollary}
\newtheorem{remark}[theorem]{Remark}

\newtheorem{observation}[theorem]{Observation}

\newcommand{\beqa}{\begin{eqnarray*}}
\newcommand{\eeqa}{\end{eqnarray*}}
\newcommand{\beqn}{\begin{eqnarray}}
\newcommand{\eeqn}{\end{eqnarray}}
\newcommand{\e}{\varepsilon}
\newcommand{\del}{\delta}

\newcommand{\ds}{\displaystyle}

\newcommand{\N}{\mathbb N}

\newcounter{cnt1}
\newcounter{cnt2}
\newcounter{cnt3}
\newcounter{cnt4}
\newcommand{\blr}{\begin{list}{$($\roman{cnt1}$)$}
{\usecounter{cnt1} \setlength{\topsep}{0pt}
\setlength{\itemsep}{0pt}}}
\newcommand{\bla}{\begin{list}{$($\alph{cnt2}$)$}
{\usecounter{cnt2} \setlength{\topsep}{0pt}
\setlength{\itemsep}{0pt}}}
\newcommand{\bln}{\begin{list}{$($\arabic{cnt3}$)$}
{\usecounter{cnt3} \setlength{\topsep}{0pt}
\setlength{\itemsep}{0pt}}}
\newcommand{\blR}{\begin{list}{$($\Roman{cnt4}$)$}
{\usecounter{cnt4} \setlength{\topsep}{0pt}
\setlength{\itemsep}{0pt}}}
\newcommand{\el}{\end{list}}
\begin{document}

\title[More on MIP and UMIP]{More on the (Uniform) Mazur Intersection Property}

\author[D. Gothwal]{Deepak Gothwal}

\address{Stat--Math Division, Indian Statistical Institute, 203, B.~T. Road, Kolkata
700108, India.}
\email{deepakgothwal190496@gmail.com}



\date{\today}

\begin{abstract}
In this paper, we introduce two moduli of w*-semidenting points and characterise the Mazur Intersection Property (MIP) and the Uniform MIP (UMIP) in terms of these moduli.

We show that a property slightly stronger than UMIP already implies uniform convexity of the dual. This may lead to a possible approach towards answering the long standing open question whether the UMIP implies the existence of an equivalent uniformly convex renorming.

We also obtain the condition for stability of the UMIP under $\ell_p$-sums.
\end{abstract}

\maketitle

\section{Introduction}

The Mazur Intersection Property (MIP) has been an extensive area of study. Defined as:
\begin{quote} \em
A Banach space $X$ is said to have the MIP if every closed bounded convex set in $X$ is the intersection of the closed balls containing it,
\end{quote}
the property brings more features to the structure of the Banach space. For instance, Giles, Gregory and Sims in \cite{GGS} showed that $X$ has the MIP if and only if the set of the w*-denting points of the unit ball $B(X^*)$ of $X^*$ is dense in the unit sphere $S(X^*)$. This initiated a great lead to a lot of analysis on the behaviour of a Banach space with the MIP.

But the uniform version of the MIP (UMIP) didn't get much attention since its introduction by Whitfield and Zizler in \cite{WZ1}. It was defined as:
\begin{quote} \em
$X$ has the UMIP if for every $\e >0$, there exists $K >0$ such that for every closed bounded and convex set $C \subseteq X$ and $p \in X$ with $diam(C) <1/\e$ and $d(p, C) \geq \e$, there is a closed ball $B[x_0, r_0]$ in $X$ such that $C \subseteq B[x_0, r_0]$, $d(p, B[x_0, r_0]) \geq \e/2$ and $r_0 \leq K$.
\end{quote}

In our earlier paper \cite{BGG}, we characterised the UMIP in terms of uniform w*-semidenting points, thus filling a long felt gap. In this paper, we continue our investigation of Banach spaces with the UMIP. We introduce two moduli of w*-semidenting points and characterise the MIP and UMIP in terms of these moduli.

One of the moduli is motivated by the modulus of denting points discussed by Dutta and Lin in \cite{DL}. They have already observed that their notion of uniform denting implies the existence of an equivalent uniformly convex renorming. We improve this result and show that $X^*$ is uniformly w*-denting actually implies $X^*$ is uniformly convex. Using this, we introduce a property slightly stronger than the UMIP (uniform version of the property ``$\mathcal{P} = \mathcal{H}$'' discussed in \cite{Gi}) which characterises the uniform convexity of $X^*$. This may lead to a possible approach towards answering the long standing open question whether the UMIP implies the existence of an equivalent uniformly convex renorming of~$X^*$.

Using the second modulus, we obtain a condition for stability of the UMIP under $\ell_p$-sums.

\section{Notations and Preliminaries}

Throughout this article, $X$ is a real Banach space. For $x\in X$ and $r>0$, we denote by $B(x, r)$ \emph{the open ball} $\{y\in X : \|x-y\|<r\}$ and by $B[x, r]$ \emph{the closed ball} $\{y\in X: \|x-y\|\leq r\}$. We denote by $B(X)$ the \emph{closed unit ball} $\{x\in X : \|x\| \leq 1\}$ and by $S(X)$ the \emph{unit sphere} $\{x \in X : \|x\| = 1\}$.

For $x\in S(X)$, we denote by $D(x)$ the set $\{f\in S(X^*) :
f(x)=1\}$. Any selection of $D$ is called a support mapping.

\begin{definition} \rm
A \emph{slice} of $B(X)$ determined by $f\in S(X^*)$ is a set of the form
\[
S(B(X), f, \alpha) := \{x\in B(X) : f(x) > \alpha \}
\]
for some $0<\alpha < 1$.

For $x \in S(X)$, $S(B(X^*), x, \alpha)$ is called a w*-slice of $B(X^*)$.

We say that $x\in S(X)$ is a \emph{denting point} of $B(X)$ if
for every $\e > 0$, $x$ is contained in a slice of $B(X)$
of diameter less than $\e$.

A w*-denting point of $B(X^*)$ is defined similarly.
\end{definition}

\begin{remark} \rm
Note that a more common notation for a slice is
\[
S(B(X), f, \alpha) := \{x\in B(X) : f(x) > 1-\alpha \}.
\]
However, the above is more convenient in our context.
\end{remark}

\begin{definition} \rm \cite{BGG}
$f\in S(X^*)$ is said to be a \emph{w*-semidenting point} of $B(X^*)$ if for every $\e > 0$, there exist $0 < \alpha < 1$ and $x \in S(X)$ such that $S(B(X^*), x, \alpha) \subseteq B(f, \e)$.

A semidenting point of $B(X)$ can be defined similarly.
\end{definition}

\begin{theorem} \label{thm0} \bla
\item \cite{CL} A Banach space $X$ has the MIP if and only if every $f \in S(X^*)$ is w*-semidenting.
\item \cite{BGG} A Banach space $X$ has the UMIP if and only if every $f \in S(X^*)$ is uniformly w*-semidenting, i.e., given $\e > 0$, there exists $0 < \alpha < 1$ such that for any $f \in S(X^*)$, there exists $x \in S(X)$ such that
\[
S(B(X^*), x, \alpha) \subseteq B(f, \e).
\]
\el
\end{theorem}

We now recall the notions of uniform convexity.

\begin{definition} \rm
For $t >0$, let us consider the following quantity:
\[
\delta_X (t):= \inf \left\{1 - \frac{\|x + y\|}{2} : x, y \in S(X), \|x - y\| \geq t\right\}.
\]

This quantity is called the modulus of convexity of $X$. A Banach space $X$ is uniformly convex if $\delta_X(t) >0$ for all $t >0$.

We will write $\del(t)$ in place of $\del_{X}(t)$ if there is no scope of confusion.
\end{definition}

\subsection{Modulus of Denting Point}

We now look at the modulus of denting and uniform denting points discussed by Dutta and Lin \cite{DL}. This serves as a background for the moduli we define later. For the sake of completeness, we include the proofs, filling up some of the details missing in \cite{DL}.

\begin{definition} \rm
For $0 <t <2$, $f \in S(X^*)$ and $x \in S(X)$, define
\[
s(x, f, t) := \inf \{\|x + y\| - 1 : y \in \ker f, \|y\| \geq t/4\}.
\]
For $0 <t <2$ and $x \in S(X)$, define
\beqa
d(x, t) & := & \sup_{f \in S(X^*)}s(x, f, t) \quad \mbox{ and}\\
d(t) & := & \inf_{x \in S(X)}d(x, t).
\eeqa
\end{definition}

Note that for $0 <t <2$, $x \in S(X)$, $f \in D(x)$ and $y \in \ker f$,
\[
1 = f(x) = f(x+y) \leq \|x+y\|.
\]
It follows that $s(x, f, t) \geq 0$ and hence, $d(x, t) \geq 0$.

\begin{lemma} \cite[Lemma 3.1]{DL} \label{lem1}
Let $x \in S(X)$ and $f \in S(X^*)$ be such that $f(x) > 0$ and for some $0 <t <2$, $s(x, f, t) >0$, then
\[
diam (S(B(X), f, f(x)(1-s(x, f, t)))) < 2t.
\]
\end{lemma}

\begin{proof}
Suppose $s = s(x, f, t) >0$. Let $v \in \ker f$ with $\|v\|=t/4$. Clearly, $s \leq \|x\| + \|v\| - 1 \leq t/4 <1$.

\textsc{Claim~:} $f(x) \geq 1 -t/4$.

Let, if possible, $f(x) < 1 -t/4$. Let $0 <\e < \min\{s/(1+s), 1- t/4 - f(x)\}$.

Let $v \in S(X)$ be such that $f(v) >1 -\e$. Put $z = x - \frac{f(x)}{f(v)}v$. Then, $z \in \ker f$ and
\[
\|z\| \geq 1 - \frac{f(x)}{f(v)} \geq 1- \frac{1-t/4-\e}{1-\e} = \frac{t/4}{1-\e} >t/4.
\]
 It follows that
 \[
 1 + s \leq \|x-z\| \leq \frac{f(x)}{f(v)}\leq \frac{1}{1-\e} < 1 + s.
 \]
 This is a contradiction. Hence, the claim.

Let $y \in S = S(B(X), f, f(x)(1-s))$. Then, $f(y) >f(x)(1-s)$. Consider $z = y - \frac{f(y)}{f(x)}x$. Then $z \in \ker f$. Also,
\[
\|x + z\| - 1 = \left\|x + y - \frac{f(y)}{f(x)}x\right\| - 1 \leq \left|1 - \frac{f(y)}{f(x)}\right|.
\]

\textsc{Case 1 :} Suppose $f(y) \leq f(x)$. Then, $|1 - f(y)/f(x)| = 1 - f(y)/f(x) < s$. So, $\|z\| <t/4$. Hence,
\[
\|y - x\| = \left\|z + \frac{f(y)}{f(x)}x -x\right\| \leq \|z\| + \left|1 - \frac{f(y)}{f(x)}\right| <t/4 + s \leq t/4 + t/4 = t/2.
\]

\textsc{Case 2 :} If $1 \geq f(y) >f(x)$, put $z_1 = f(x)z = f(x)y - f(y)x$. Then $z_1 \in \ker f$. Also,
\[
\|x + z_1\| - 1 = \|x + f(x)y - f(y)x\| - 1 \leq f(x) + 1 - f(y) - 1 \leq 0 <s.
\]

 Then, $\|f(x)z\| < t/4$, i.e., $\|z\| < (t/4)/f(x)$. Therefore, by the claim above,
\[
\|y - x\| \leq \|z\| + \left|1 - \frac{f(y)}{f(x)}\right|
< \frac{t/4}{f(x)} + \frac{1}{f(x)} - 1
\leq \frac{t/4 + 1}{1 - t/4} - 1 < t.
\]
Since $x \in S$, $diam (S) < 2t$.
\end{proof}

\begin{lemma} \label{lem0}
Given $x \in S(X)$ and $0<t<2$, suppose there exists $f \in S(X^*)$ and $\alpha \in (0, 1)$ such that $x\in S:=S(B(X), f, \alpha)$ and $diam (S) <t/5$. Then
\[
s(x, f, t) \geq \beta := \min\left\{\frac{f(x) - \alpha}{\alpha}, \frac{t}{20}\right\} > 0.
\]
\end{lemma}

\begin{proof}
Let $y \in \ker f$ such that $z = x + y \in B(X)$. Then $z \in S$. Thus, $\|y\| = \|z - x\| < t/5$. This implies $s(x, f, t)\geq 0$.

Suppose $0 \leq s(x, f, t) < \beta$. Choose $\e >0$ such that $s(x, f, t) < \e < \beta$.

Then $1 \geq f(x) >\alpha (1+\e) >\alpha$.

Clearly, $x \in S_1:=\{y \in B(X) : f(y) >\alpha (1+\e)\} \subseteq S$ and so, $diam S_1 < t/5$.

Now, since $0 \leq s(x, f, t) < \e$, there exists $y_0 \in \ker f$ such that $\|y_0\|\geq t/4$ and $1 \leq \|x+y_0\|<1+\e$.

Put $z_0 = x+y_0$. Then, $f(z_0) >\alpha(1+\e)$. Therefore, $f(z_0/\|z_0\|)>\alpha$.

This implies, $z_0/\|z_0\| \in S$. So, $\|x - z_0/\|z_0\| \| < t/5$. Therefore,
\[
\|y_0\| = \|x - z_0\| \leq \left\|x - \frac{z_0}{\|z_0\|}\right\| + \left\|z_0 - \frac{z_0}{\|z_0\|}\right\| < \frac t5 + (\|z_0\| -1) < \frac t5 +\e < \frac t4.
\]
This is a contradiction, proving $s(x, f, t) \geq \beta$.
\end{proof}

\begin{theorem} \cite[Proposition 3.3 $(a)$]{DL} \label{thm1}
$x \in S(X)$ is a denting point of $B(X)$ if and only if $d(x, t) >0$ for all $0 <t <2$.
\end{theorem}

\begin{proof} Let $0 <t <2$ be given. Suppose $d(x, t) >0$. Then there exists $f \in S(X^*)$ such that $s = s(x, f, t) >0$.

If $f(x) = 0$, then $(t/4)(-x) \in \ker f$. So, $s \leq \|x - (t/4)x\| -1 \leq -t/4 <0$, a contradiction. Hence, $f(x) \neq 0$. Note that $s(x, f, t) = s(x, -f, t)$. Therefore, without loss of generality, we may assume $f(x) > 0$. Hence, by Lemma~\ref{lem1}, $x \in S = S(B(X), f, f(x)(1-s))$ and $diam (S) < 2t$. Since $0 <t <2$ is arbitrary, $x$ is a denting point of $B(X)$.

Conversely, let $0<t<2$ be given. Since, $x$ is denting point, there exists $f \in S(X^*)$ and $\alpha \in (0, 1)$ such that $x\in S:=S(B(X), f, \alpha)$ and $diam (S) <t/5$.

By Lemma~\ref{lem0}, $s(x, f, t) \geq \min\{(f(x) - \alpha)/\alpha, t/20\} > 0$ which implies that $d(x, t)>0$.
\end{proof}

\begin{definition} \rm
A Banach space $X$ is uniformly denting if $d(t) >0$ for all $0 <t <2$.
\end{definition}

Dutta and Lin \cite{DL} sketch an argument to show that if $X$ is uniformly denting then $X$ has an equivalent uniformly convex norm. We will elaborate and improve on this in Section~\ref{s4}.

We end this section with another observation regarding the above modulus.
\begin{proposition} \label{prp2}
If $D$ is a dense set in $S(X)$, then
\[
d(t) = \inf_{x\in D}d(x, t) \mbox{ for all } 0<t<2.
\]
\end{proposition}

\begin{proof}
Let $d_D(t):=\inf_{x\in D}d(x, t)$. Clearly, $d_D(t)\geq d(t)$ for all $0<t<2$.

Let $\e>0$ and $x \in S(X)$. Since $D$ is dense in $S(X)$, there exists $z \in D$ such that $\|x-z\|<\e$. Then for any $f \in S(X^*)$ and $y \in \ker f$ with $\|y\| \geq t/4$, we have $\|z+y\| \leq \|x +y\| + \e$. It follows that $s(z, f, t) \leq s(x, f, t) + \e$ and hence, $d_D(t)\leq d(z, t) \leq d(x, t) + \e$. Therefore, $d_D(t)\leq d(t) + \e$.
\end{proof}

\section{Modulus of w*-Semidenting Points}

In this section, we are going to study the MIP and the UMIP with the help of some moduli motivated by the above moduli.

We begin with the analogues of the moduli defined above for dual spaces.

\begin{definition} \rm
For $0 <t <2$, $f \in S(X^*)$ and $x \in S(X)$, define
\[
s^*(f, x, t) := \inf \{\|f+h\|-1 : \|h\|\geq t/4, h(x)=0\}.
\]
For $0 <t <2$ and $f \in S(X^*)$, define
\beqa
d^*(f, t) & := & \sup_{x \in S(X)} s^*(f, x, t) \quad \mbox{ and}\\
d^*(t) & := & \inf_{f \in S(X^*)} d^*(f, t).
\eeqa
\end{definition}

The following lemma is the w*-analogue of Lemma \ref{lem1} and Theorem ~\ref{thm1}.
\begin{lemma} \label{lem7}
\bla
\item Let $x \in S(X)$ and $f \in S(X^*)$ be such that $f(x) > 0$ and for some $0 <t <2$, $s^*(f, x, t) >0$, then
\[
diam (S(B(X^*), x, f(x)(1-s^*(f, x, t)))) < 2t.
\]

\item $f \in S(X^*)$ is a w*-denting point of $B(X^*)$ if and only if $d^*(f, t) >0$ for all $0 <t <2$.
\el
\end{lemma}

\begin{definition} \rm
$X^*$ is uniformly w*-denting if $d^*(t) >0$ for all $0 <t <2$.
\end{definition}

The proof of $(b)$ below is essentially contained in that of \cite[Fact 2.2]{DHR}. We include the proof for completeness.

\begin{lemma} \label{lem8}
Let $\e > 0$, $0 < \beta < \alpha < 1$ and $x \in S(X)$. \bla
\item If $diam (S(B(X^*), x, \alpha)) < \e$, then $\alpha \geq 1 - \e$.
\item $\ds diam (S(B(X^*), x, \beta)) \leq \left(\frac{1-\beta}{1-\alpha}\right) diam (S(B(X^*), x, \alpha))$.
\el
\end{lemma}

\begin{proof}
$(a)$. Let $h \in D(x)$. If possible, let $\alpha < 1-\e$. Choose $g \in B(X^*)$ such that $\alpha < g(x) < 1-\e$. Then $h, g \in S(B(X^*), x, \alpha)$ but $\|h-g\| \geq h(x) - g(x) > \e$.

$(b)$. Let $g, h \in S(B(X^*), x, \beta)$, $f \in D(x)$ and $\lambda = \frac{\alpha - \beta}{1-\beta}$.
Let $g_{\lambda}:=\lambda f + (1-\lambda)g$ and $h_{\lambda}:=\lambda f + (1-\lambda)h$.
We have,
\[
g_{\lambda}(x)> \lambda + (1-\lambda) \beta = \alpha.
\]
So, $g_{\lambda} \in S(B(X^*), x, \alpha)$. Similarly, $h_{\lambda} \in S(B(X^*), x, \alpha)$. Therefore,
\beqa
\|h-g\| & = & \frac{\|g_{\lambda} - h_{\lambda}\|}{1-\lambda} \leq \frac{diam (S(B(X^*), x, \alpha))}{1-\lambda}\\ & = & \left(\frac{1-\beta}{1-\alpha}\right) diam (S(B(X^*), x, \alpha)).
\eeqa
Hence,
\[
diam (S(B(X^*), x, \beta)) \leq \left(\frac{1-\beta}{1-\alpha}\right) diam (S(B(X^*), x, \alpha)).
\]
\end{proof}

\begin{corollary} \label{cor2}
Let $\gamma >0$ and $k \in \N$ be such that $1-2k(1-\gamma) > 0$, then
\[
diam (S(B(X^*), x, 1-2k(1-\gamma))) \leq 2k (diam(S(B(X^*), x, \gamma))).
\]
\end{corollary}

\begin{theorem} \label{thm}
For a Banach space $X$, the following are equivalent~: \bla
\item For every $t >0$, there exists $\delta >0$ such that for every $f \in S(X^*)$, there exists $x \in S(X)$ satisfying
\[
diam (S(B(X^*), x, f(x)-\delta)) <t.
\]
\item For every $t >0$, there exists $\del >0$ such that for every $f \in S(X^*)$, there exists $x \in S(X)$ satisfying
\[
diam(S(B(X^*), x, f(x)(1-\del))) <t.
\]
\item $X^*$ is uniformly w*-denting.
\el
\end{theorem}

\begin{proof}
$(a) \implies (b)$.
Let $t >0$ be given. So, there exists $\delta >0$ such that for every $f \in S(X^*)$, there exists $x \in S(X)$ satisfying
\[
diam (S(B(X^*), x, f(x)-\delta)) <t.
\]
Since, $f(x)(1-\delta) \geq f(x)-\delta$,
\[
diam (S(B^*, x, f(x)(1-\delta))) <t.
\]

$(b) \implies (c)$.
Let $t >0$ be given. By $(b)$, there exists $\del > 0$ that works for $t/5$. Let $f \in S(X^*)$. Then, there exists $x \in S(X)$ satisfying
\[
diam(S(B(X^*), x, f(x)(1-\del))) <t/5.
\]
By Lemma~\ref{lem0}, $s^*(f, x, t) \geq \beta(t) := \min\{\del/(1-\del), t/20\} > 0$, which implies that $d^*(t) \geq \beta(t) >0$. Hence, $X^*$ is uniformly w*-denting.

$(c) \implies (a)$.
Let $t >0$ be given. Then, $\delta:=d^*(t/4)/2 >0$. So, by definition, for every $f \in S(X^*)$, there exists $x \in S(X)$ such that $s^*(f, x, t/4) >\delta$. Therefore, by Lemma~\ref{lem7}(a),
\[
diam (S(B(X^*),x,f(x)(1-\delta)))<t/2.
\]

By Corollary~\ref{cor2},
\[
diam (S(B(X^*), x, 1-2(1-f(x)(1-\delta)))) <t.
\]
Now,
\[
f(x)-\delta + 1 \geq f(x)(1-\delta) + f(x)(1-\delta) = 2f(x)(1-\delta).
\]
This implies that
\[
f(x)-\delta \geq 2f(x)(1-\delta)-1 = 1-2(1-f(x)(1-\delta)).
\]

Hence,
\[
diam (S(B(X^*), x, f(x)-\delta)) <t.
\]
\end{proof}

\begin{definition} \rm
For $0 <t <2$ and $f \in S(X^*)$, define
\beqa
d^*_0(f, t) & := & \sup \{s^*(g, x, t) : x\in S(X), \, g\in S(X^*), \, \|f-g\|<t\}\\
& = & \sup \{d^*(g, t) : g\in S(X^*), \, \|f-g\|<t\} \quad \mbox{ and}\\
d_0^*(t) & := & \inf_{f \in S(X^*)} d_0^*(f, t).
\eeqa
\end{definition}

\begin{theorem} \label{thm4}
\bla
\item $f\in S(X^*)$ is a w*-semidenting point of $B(X^*)$ if and only if for any $0 <t <2$, $d^*_0(f, t) > 0$.
\item $X$ has the MIP if and only if $d^*_0(f, t) >0$ for every $f \in S(X^*)$ and $t \in (0, 2)$.
\item $X$ has the UMIP if and only if for any $t \in (0, 2)$,
 $d^*_0(t) > 0$.
\el
\end{theorem}

\begin{proof}
$(a)$. Let $0 <t <2$ be given. Since $d^*_0(f, t/4) > 0$, there exist $x\in S(X)$ and $g\in S(X^*)$ such that $\|f - g\| <t/4$ and $\ds s^*(g, x, t/4) > 0$.
As before, $g(x) \neq 0$ and we may assume $g(x) > 0$.
By Lemma ~\ref{lem1}, we have a w*-slice $S$ of diameter smaller than $t/2$ containing $g$.
Now, $\|f - g\| <t/2$. Hence, we have $S \subseteq B(f, t)$.
Hence, $f$ is a w*-semidenting point.

Conversely, let $0 <t <2$ and $f$ be a w*-semidenting point. So, there exist $0 <\alpha <1$ and $x \in S(X)$ such that $S = S(B(X^*), x, \alpha) \subseteq B(f, t/10)$.
Choose $g_0 \in S$ such that $g_0 \in S(X)$. Then $g_0 \in S$ and $diam(S) < t/5$. By Lemma~\ref{lem0}, $s^*(g_0, x, t) >0$ and $\|g_0-f\| < t/10 <t$. Hence, $d^*_0(f, t) >0$.

$(b)$ follows from $(a)$ and Theorem~\ref{thm0} $(a)$.

$(c)$. Let $X$ has the UMIP. So, by Theorem~\ref{thm0} $(b)$, for any $t \in (0, 2)$, there exists $0 <\alpha(t/10) < 1 $ such that for every $f\in S(X^*)$, there is an $x\in S(X)$ such that
\[
S(B(X^*), x, \alpha(t/10)) \subseteq B(f, t/10).
\]
We have, $diam(S) <t/5$.
Choose $g_0 \in D(x)$. Then $g_0 \in S$. By Lemma~\ref{lem0}, $\ds s^*(g_0, x, t) \geq \beta(t) : = \min\left\{\frac{1}{\alpha(t)} -1, \frac{t}{20}\right\} > 0$.
Now, since $g_0 \in S$, $\|f - g_0\| <t$, and hence, $d^*_0(f, t) \geq \beta(t)$. It follows that $d^*_0(t) \geq \beta(t) > 0$.

Conversely, let $0 <t <2$. Let $\gamma(t) = \frac12 d^*_0(t/4) > 0$.
Let $f \in S(X^*)$. Then $d^*_0(f, t/4) \geq d^*_0(t/4) > \gamma(t)$.
So, there exists $x\in S(X)$ and $g\in S(X^*)$ such that $\|f-g\|<t/4$ and $\ds s^* = s^*(g, x, t/4) > \gamma(t)$. As before, changing $x$ to $-x$ if necessary, we may assume $g(x) > 0$.
Let $S := S(B(X^*), x, g(x) (1- s^*))$. Then $g \in S$ and by Lemma~\ref{lem1}, we have $diam(S) < t/2$. Since $\|f-g\|<t/4$, we have $S \subseteq B(f, t)$.
Also, $0< g(x) \leq 1$ and $\ds s^* > \gamma(t)$ imply that
\[
S(B(X^*), x, 1- \gamma(t)) \subseteq S(B(X^*), x, 1- s^*) \subseteq S \subseteq B(f, t).
\]
It follows that $\alpha(t) =1-\gamma(t)$ works.
\end{proof}

\begin{remark} \rm
Note that if $X$ has the UMIP, it is uniformly w*-semidenting, also since it has the MIP, the set $D$ of w*-denting points of $B(X^*)$ is norm dense in $S(X^*)$. However, it does not follow that these w*-denting points are uniform over $D$, that is, $\inf_{f\in D} d^*(f, t) > 0$, as this, by Proposition ~\ref{prp2}, would imply that $X^*$ is uniformly w*-denting, which, in general, is a much stronger condition, as finite dimensional examples can show.
\end{remark}

\section{Uniform w*-Denting and Hyperplane UMIP} \label{s4}

Now we will discuss a geometric property which is a stronger form of the UMIP, a uniform version of $``\mathcal{P}=\mathcal{H}"$ discussed in \cite{Gi}. We will call it Hyperplane UMIP (H-UMIP). We will show that this property is equivalent to $X^*$ being uniformly convex.

First, let us recall the definition of the condition ``$\mathcal{P}=\mathcal{H}$".

\begin{definition} \rm
We say that a Banach space $X$ has the Hyperplane MIP (H-MIP) if for every closed convex bounded set $C \subseteq X$ and $f \in S(X^*)$ such that $\inf f(C) >0$, there exists a closed ball $B[x_0, r_0]$ in $X$ such that $C \subseteq B[x_0, r_0]$ and $\inf f(B[x_0, r_0]) >0$.
\end{definition}

The following characterisation of the above property is already known:
\begin{theorem} \cite[Proposition 5.3, p. 196]{CL1}
$X$ has the H-MIP if and only if every $f \in S(X^*)$ is a w*-denting point of $B(X^*)$.
\end{theorem}

\begin{definition} \rm
We say that $X$ has the Hyperplane UMIP (H-UMIP) if for every $\e >0$ and $M(\e) \geq 1$, there is $K(\e) > 0$ such that whenever a closed convex set $C\subseteq X$ and $f \in S(X^*)$ are such that $\sup\{\|x\| : x \in C\} \leq M(\e)$ and $\inf f(C) \geq \e$, there is a closed ball $B[x_0, r_0]$ in $X$ such that $C \subseteq B[x_0, r_0]$, $\inf f(B[x_0, r_0]) \geq \e/2$ and $r_0 \leq K$.
\end{definition}

\begin{remark} \rm
Note that unlike the UMIP, in this case, the uniform bound is not on the diameter of $C$, but its farthest distance from the point being separated, which is an apparently stronger assumption.
\end{remark}

\begin{definition} \rm
$X^*$ is said to be uniformly w*-denting (I) if for every $\e >0$, there exists $0 < \alpha < 1$ such that for every $f \in S(X^*)$, there is $x \in S(X)$ such that
\[
f(x) > \alpha \mbox{ and } diam (S(B(X^*), x, \alpha)) <\e.
\]
\end{definition}

\begin{lemma} \label{lem4}
If $X^*$ is uniformly w*-denting (I), then for every $\e >0$, there exist $0 < \beta < \alpha < 1$ such that for every $f \in S(X^*)$, there exists $x \in S(X)$ satisfying
\[
diam (S(B(X^*), x, \beta)) <\e \mbox{ and } f(x) >\alpha.
\]
\end{lemma}

\begin{proof}
Let $X^*$ be uniformly w*-denting (I). Let $\e >0$ be given. There exists $\alpha >0$ such that given $f \in S(X^*)$, there exists $x \in S(X)$ such that
\[
diam (S(B(X^*), x, \alpha)) <\e/2 \mbox{ and } f(x) >\alpha.
\]
Let $\beta > 0$ be such that $\alpha > \beta > 2\alpha -1$. By Lemma \ref{lem8} $(b)$, we have that
\[
diam (S(B(X^*), x, \beta)) \leq \left(\frac{1-\beta}{1-\alpha}\right) diam (S(B(X^*), x, \alpha)) < \e.
\]
\end{proof}

For $\e >0$, let
\[
F_{\e}: = B(X^*) \setminus \bigcup\{S : S \mbox{ is a w*-slice of } B(X^*) \mbox{ with } diam(S) < \e\}.
\]

\begin{theorem} \label{thm5}
The following statements are equivalent~: \bla
\item $X^*$ is uniformly w*-denting.
\item For every $\e >0$, there exists $\del(\e) >0$ such that
\[
F_{\e} \subseteq (1-\del(\e))B(X^*).
\]
\item $X^*$ is uniformly convex.
\item $X^*$ is uniformly w*-denting (I).
\item $X$ has the H-UMIP.
\el
\end{theorem}

\begin{proof}
$(a) \implies (b)$.
Let $X^*$ be uniformly w*-denting and $\e >0$. By Theorem ~\ref{thm}, there exists $\del >0$ such that for every $f \in S(X^*)$, there is $x \in S(X)$ such that
\[
diam (S(B(X^*), x, f(x)(1-\del))) <\e.
\]

Let $f \in F_{\e}$. Now, for $h = f/\|f\|$, there exists $x \in S(X)$ such that
\[
diam (S(B(X^*), x, h(x)(1-\del))) <\e.
\]

Since, $f \notin S(B(X^*), x, h(x)(1-\del))$, $f(x) \leq h(x)(1-\del)$. Then, $\|f\| \leq 1-\del$. Therefore,
\[
F_{\e} \subseteq (1-\del) B(X^*).
\]

$(b) \implies (c)$. Let $\e >0$ be given. Let $\del > 0$ be such that
\[
F_{\e/2} \subseteq (1-\del) B(X^*).
\]

\textsc{Claim~:} If $f, g \in S(X^*)$ and $\|f-g\| \geq \e$, then $\ds \left\|\frac{f+g}{2}\right\| \leq 1-\del$.

Suppose not. Then, $\ds \left\|\frac{f+g}{2}\right\| > 1-\del$ and hence, $\ds \frac{f+g}{2} \notin F_{\e/2}$.

Thus, there exists $S = S(B(X^*), x, \alpha)$ such that $\ds \frac{f+g}{2} \in S$ and $diam(S) < \e/2$.

Now, $\ds \frac{f(x)+g(x)}{2} > \alpha$ implies either $f(x) > \alpha$ or $g(x) > \alpha$, that is, at least one of $f$ or $g \in S$.

Since $\ds \frac{f+g}{2} \in S$ and $diam(S) < \e/2$, $\ds \left\|\frac{f-g}{2}\right\| < \e/2$, or $\ds \left\|f-g\right\| < \e$. Contradiction.
So, $X^*$ is uniformly convex.

$(c) \implies (d)$.
Let $\e >0$ be given. So, there exists $0 <\del \leq \e/4$ such that for any $f, g \in S(X^*)$ with
\[
\frac{\|f+g\|}{2} >1-\del,
\]
we have $\|f-g\| \leq \e/4$.

Let $f \in S(X^*)$. Choose $x \in S(X)$ such that $f(x)=1$ (this is possible since, because of the uniform convexity of $X^*$, $X$ is reflexive). Let $g \in B(X^*)$ be such that $g(x) > 1-\del$. We have $\|g\| >1-\del \geq 1-\e/4$ and therefore, $1-\|g\| <\e/4$ and $\frac{g}{\|g\|}(x) > \frac{1}{\|g\|}(1-\del) >1-\del$.

Now,
\beqa
\frac{\|f+(g/\|g\|)\|}{2} & \geq & \frac{(f+(g/\|g\|))(x)}{2}=\frac{1}{2}(1+(g/\|g\|)(x))\\
& > & \frac{1}{2}(1+1-\del) = 1-\del/2 >1-\del.
\eeqa
Then, $\|f-\frac{g}{\|g\|}\| <\e/4$.
\[
\|f-g\|=\|f-\frac{g}{\|g\|}\|+1-\|g\| <\e/2.
\]

Hence, $diam (S(B(X^*), x, 1-\del)) <\e$ and $f \in S(B(X^*), x, 1-\del)$. So, $X^*$ is uniformly w*-denting (I).

$(d) \implies (e)$.
Let $\e >0$ and $M \geq 2$ be given. Let $M_1 := M + 3\e/4$.
Since $X^*$ is uniformly w*-denting (I), by Lemma~\ref{lem4}, there exist $0 < \alpha < 1$ and $\del >0$ such that for every $f \in S(X^*)$, there exists $x \in S(X)$ satisfying
\[
diam (S(B(X^*), x, \alpha - \del)) <\e/4M_1 \mbox{ and } f(x) >\alpha.
\]

Choose $k \in \mathbb{N}$ such that $1/2k <\e/4M_1$. Again, there exists $0 < \gamma < 1$ such that for every $f \in S(X^*)$, there is $x \in S(X)$ such that
\[
f(x) > \gamma \mbox{ and } diam (S(B(X^*), x, \gamma)) <\del/4k.
\]

Then, by Lemma~\ref{lem8} $(a)$, $(1 - \gamma) \leq \del/4k <1/2k <\e/4M_1$. So, $1-2k(1-\gamma) >0$.

It suffices to show that for $K=\frac{M_1}{2k(1 - \gamma)}$, whenever a closed convex set $C$ and $f \in S(X^*)$ are such that $\sup\{\|x\| : x \in C\} \leq M$ and $\inf f(C) \geq \e$, there is a closed ball $B[x_0, r_0]$ in $X$ such that $C \subseteq B[x_0, r_0]$, $\inf f(B[x_0, r_0]) \geq \e/2$ and $r_0 \leq K$.

Let $C$ and $f \in S(X^*)$ be as above. For $f$, there exists $x_1 \in S(X)$ such that
\[
diam(S(B(X^*), x_1, \alpha-\del)) <\e/4M_1 \mbox{ and } f(x_1) >\alpha.
\]

Again, there exists $x_2 \in S(X)$ such that
\[
diam (S(B(X^*), x_2, \gamma)) <\del/4k \mbox{ and } f(x_2) >\gamma.
\]

If $g \notin S(B(X^*), x_1, \alpha-\del/2)$, then $\|f-g\| \geq f(x_1) - g(x_1) > \del/2$.

Let $\eta = 1-2k(1-\gamma) > 0$. By Corollary ~\ref{cor2},
\[
diam (S(B(X^*), x_2, \eta)) \leq 2k (diam (S(B(X^*), x_2, \gamma))) < \del/2.
\]
Since $f \in S(B(X^*), x_2, \eta)$, it follows that $g \notin S(B(X^*), x_2, \eta)$. Thus, we obtain that
\beqa
f \in S(B(X^*), x_2, \gamma) & \subseteq & S_1 = S(B(X^*), x_2, \eta) \\
& \subseteq & S(B(X^*), x_1, \alpha-\del/2) \subseteq S(B(X^*), x_1, \alpha-\del).
\eeqa
So, $diam (S_1) <\e/4M_1$.

Define $D = \overline{C + 3\e/4 B(X)}$. Then $D \subseteq B[0, M_1]$ and $\inf f(D) \geq \e/4$. Hence, for $z \in D$ and $g \in B(X^*) \cap B(f, \e/4M_1)$,
\[
g(z) \geq f(z) - \|f-g\| \|z\| > \e/4 - \frac{\e M_1}{4M_1}=0.
\]
Therefore,
\[
S_1 \subseteq B(X^*) \cap B(f, \e/4M_1) \subseteq \{g \in B(X^*) : g(z) > 0 \mbox{ for all } z \in D\}.
\]

Let
\[
d = \frac{d(0, D)(1-\eta)}{1+\eta} > 0.
\]
Then as in \cite[Theorem 2.6]{BGG}, for $\lambda = M_1/(1-\eta) = M_1/2k(1-\gamma)$,
\[D \subseteq B \left[\lambda x_2, \lambda - d \right].
\]
i.e.
\[
C \subseteq B \left[\lambda x_2, \lambda - d - \frac{3\e}{4}\right] = B_1.
\]

Since, $1/2k < \e/4M_1$, i.e., $M_1/2k < \e/4$, we have
\beqa
\inf f (B_1) & = & f(\lambda x_2) - \lambda + d + \frac{3\e}{4} > \lambda (\gamma - 1) + d + \frac{3\e}{4}\\
& = & -\frac{M_1}{2k} + d + \frac{3\e}{4} \geq d + \frac{3\e}{4}-\e/4 >\e/2.
\eeqa

Also, \[
\frac{M_1}{2k(1 - \gamma)}- d - \frac{3\e}{4} \leq \frac{M_1}{2k(1 - \gamma)} = K.
\]
Hence, $X$ has the H-UMIP.

$(e) \implies (a)$.
Let $0< \e < 1$ and $M(\e) \geq 1$ be given. Then, there is $K = K(\e) > 0$ such that whenever a closed convex set $C\subseteq X$ and $f \in S(X^*)$ are such that $\sup\{\|x\| : x \in C\} \leq M(\e)$ and $\inf f(C) \geq \e/4$, there is a closed ball $B[x_0, r_0]$ in $X$ such that $C \subseteq B[x_0, r_0]$, $\inf f(B[x_0, r_0]) \geq \e/8$ and $r_0 \leq K$. We may assume w.l.o.g.\ that $K \geq 1$.

Let $f \in S(X^*)$. Consider $A := \{x\in B(X) : f(x) \geq \e/4\}$. Then, $\inf f(A) = \e/4$. Also, $\sup\{\|x\| : x \in C\} \leq 1 \leq M(\e)$. So, there exist $B = B[x_0, r]$ containing $A$ such that $r \leq K$ and $\inf f(B) \geq \e/8$. Put $\alpha = K/(K+\e/9)$. We claim that $\del = 1- \alpha$ works.
Note that $r/(r+\e/9) \leq K/(K+\e/9) = \alpha$.

Since $\inf f(B) \geq \e/8$ we have $f(x_0) \geq r + \e/8 > r + \e/9$.

Put $z_0 = x_0/\|x_0\|$. Let $S = S(B(X^*), z_0, f(z_0) (1-\delta))$.

Let $g \in S$. Then, $g(z_0) > f(z_0) \alpha$, which implies, $g(x_0) > f(x_0) \alpha > f(x_0)(r/(r+\e/9)) \geq r$. That is, $\inf g(B) > 0$. So,
\[
\{x\in B(X) : f(x) > \e/3\} \subseteq A \subseteq B \subseteq \{x: g(x) > 0\}.
\]
So, by \cite[Lemma 2.5]{BGG}, $\|f-g\| \leq 2\e/3$.
But $f \in S$.
Therefore,
\[
diam (S) < 4\e/3.
\]
So, $X^*$ is uniformly w*-denting.
\end{proof}

\begin{remark} \rm
$(a) \implies (b)$ above is hinted in \cite[p 11]{DL}, but the details are from a personal communication from Late Prof.\ Sudipta Dutta.

$(b) \implies (c)$ was pointed out to us by Prof.\ Gilles Lancien recently.
\end{remark}

The following quantitative estimates follow from the proof of Theorem~\ref{thm5}.

\begin{corollary} \label{ob1} \rm
If $d^*(t) >0$ for every $t >0$, then, for every $0 <t <2$,
\bla
\item $\del_{X^*}(t) \geq d^*(t/4)$ and 
\item $d^*(t)\geq \del_{X^*}(t/20)$.
\el
\end{corollary}

\begin{proof}
If $d^*(t) >0$ for every $t >0$, by Theorem~\ref{thm5}, $\del_{X^*}(t) >0$ for every $t >0$. In particular, $X$ is reflexive.

$(a)$. As in the proof of $(a) \implies (b)$ above, if $d^*(t) >0$ for every $t >0$, for every $f \in S(X^*)$, there is $x \in S(X)$ such that
\[
diam (S(B(X^*), x, f(x)(1-d^*(t/4)))) <t/2.
\]
and hence, 
\[
F_{t/2} \subseteq (1-d^*(t/4))B(X^*).
\]
Now, the proof of $(b) \implies (c)$ above gives us that for any $f, g \in S(X^*)$ with
$\|(f+g)/2\| >1-d^*(t/4)$ we have $\|f-g\| <t$.

Hence, $\del_{X^*}(t) \geq d^*(t/4)$.

$(b)$. For notational convenience, let us put $\delta = \del_{X^*}(t/20)$. So, for any $f, g \in S(X^*)$ with
\[
\frac{\|f+g\|}{2} >1-\del,
\]
we have $\|f-g\| \leq t/20$.

Let $f \in S(X^*)$ and $x \in S(X)$ be such that $f(x)=1$. As in the proof of $(c) \implies (d)$ above, it follows that
$diam (S(B(X^*), x, 1-\delta)) < t/5$ and $f \in S(B(X^*), x, 1-\delta)$.

It is not hard to observe that $\del_{X^*}(t) \leq t$ for all $t >0$.
Since $0 <\delta \leq t/20 <1$, by Lemma ~\ref{lem0}, we have
\[
d^*(f, t) \geq s^*(f,x,t) \geq \min\left\{\frac{{\delta}}{1-\delta},t/20\right\} \geq \delta.
\]

Since, this holds for all $f \in S(X^*)$, we have $d^*(t)\geq \delta$.
\end{proof}

\section{Modulus of the UMIP by Slices}
In this section, we define another modulus of w*-semidenting points and characterise MIP and UMIP in its terms.

\begin{definition} \rm
For $f \in S(X^*)$, $x \in S(X)$ and $0 < t < 1$, define
\beqa
\beta(f, x, t) & := & \inf\{1- g(x) : g \in B(X^*), \|f-g\| \geq t\},\\
\beta(f, t) & := & \sup_{x \in S(X)} \beta(f, x, t) \mbox{ and, }\\
\beta(t) & := & \inf_{f \in S(X^*)} \beta(f, t).
\eeqa
\end{definition}

\begin{lemma} \label{lem6}
Let $x \in S(X)$, $f \in S(X^*)$ and $0 <t <1$ be such that $\beta(f, x, t) >0$, then
\[
S(B(X^*), x, 1-\beta(f, x, t)) \subseteq B(f, t)
\]
\end{lemma}

\begin{proof}
Let $g \in S(B(X^*), x, 1-\beta(f, x, t))$. Then $1 - g(x) < \beta(f, x, t)$. So, by definition of $\beta(f, x, t)$,
$\|f - g\| < t$.
\end{proof}

\begin{theorem} \label{thm7}
$X$ has MIP if and only if for every $f \in S(X^*)$ and $0 <t <1$, $\beta(f, t) >0$.
\end{theorem}

\begin{proof}
Let $f \in S(X^*)$. If for every $0 <t <1$, $\beta(f, t) >0$, then there exists $x \in S(X)$ such that $\beta(f, x, t) >0$.
So, by the above lemma, $f \in S(X^*)$ is a w*-semidenting point of $B(X^*)$. Thus, $X$ has MIP.

Conversely, let $X$ has MIP. Then, every $f \in S(X^*)$ is a w*-semidenting point of $B(X^*)$. Hence, for every $f \in S(X^*)$ and $t >0$, there exist $1 > \alpha >0$ and $x \in S(X)$ such that $S(B(X), x, \alpha) \subseteq B(f, t)$. So, if $g \in B(X^*)$ and $\|g - f\| \geq t$, then $g(x) \leq \alpha$, i.e., $0 <1-\alpha \leq 1 - g(x)$.

So, $\beta(f, x, t) = \inf\{1- g(x) : g \in B(X^*), \|f-g\| \geq t\} \geq 1-\alpha >0$. Thus, $\beta(f, t) >0$.
\end{proof}

\begin{theorem} \label{thm8}
$X$ has UMIP if and only if for every $0 <t <1$, $\beta(t) >0$.
\end{theorem}

\begin{proof}
Let $X$ has UMIP. Let $t >0$. There exists $0 <\alpha <1$ such that for any $f \in S(X^*)$, there exists $x \in S(X)$ such that if $g \in B(X^*)$ and $g(x) >\alpha$, then $\|f - g\| < t$. Hence, for every $f \in S(X^*)$, there exists $x \in S(X)$ such that $\beta(f, x, t) \geq 1-\alpha$. Then, $\beta(t) \geq 1-\alpha >0$.

Conversely, let $\beta(t) >0$ for every $0 <t <1$. So, for every $f \in S(X^*)$, there exists $x \in S(X)$ such that $\beta(f, x, t) \geq \beta(t)/2 >0$. Hence, $g \in S(B(X^*), x, 1- \beta(t)/2)$ implies $1 - g(x) < \beta(f, x, t)$. So,
\[
S(B(X^*), x, 1- \beta(t)/2) \subseteq B(f, t).
\]
Thus $X$ has the UMIP.
\end{proof}

\begin{remark} \label{rmk1} \rm
It is to be noted that for $0 <\e_1 <\e_2$, $\beta(\e_1) \leq \beta(\e_2)$.
\end{remark}

It is proved in \cite{WZ1} that if $X^*$ is uniformly convex, then $X$ has the UMIP. In the next theorem, we obtain a relationship between the modulus of uniform convexity and the modulus of UMIP when $X^*$ is uniformly convex.
\newpage

Following observation will be used in the proof in the next section.
\begin{observation} \label{ob2}
If $X^*$ is uniformly convex, then given any $f \in S(X^*)$ and $x \in D^{-1}(f)$,
\[
diam (S(B(X^*), x, 1-2\delta_{X^*}(t/2))) <t.
\]
\end{observation}

\begin{proof}
Since, $X^*$ is uniformly convex, $X$ is reflexive. Let $t >0$, then $\del_{X^*}(t) >0$. Let $f \in S(X^*)$ and $x \in S(X)$ be such that $f(x)=1$.

Let $g \in B(X^*)$ be such that $g(x) >1-2\del_{X^*}(t/2)$. Then
\[
\frac{\|f+g\|}{2} \geq \frac{f(x)+g(x)}{2} > 1 - \del_{X^*}(t/2).
\]
It follows that $\|f-g\| <t/2$.
So, $diam S(B(X^*), x, 1-2\delta_{X^*}(t/2)) <t$.
\end{proof}

\begin{remark}
It should be noted from the above observation that, if $X^*$ is uniformly convex, then $\beta(t) \geq 2\del_{X^*}(t)$.
\end{remark}
\section{Stability of the UMIP under $\ell_p$ sum}

Let $1 <p, q <\infty$ be such that $1/p + 1/q = 1$. Let $\{X_i\}$ is a sequence of Banach spaces, and
\[
X_p = \left(\bigoplus_{i=1}^\infty X_i\right)_p := \left\{(x_i) : x_i \in X_i \mbox{ such that } \sum_{i=1}^{\infty} \|x_i\|^p <\infty\right\}
\]
with $\ds \|x\|_p = \left(\sum_{i=1}^{\infty} \|x_i\|^p\right)^{1/p}$ as the norm on $X_p$.

If $X_i = X$ for all $i$, we write $\ell_p(X)$ instead of $(\bigoplus_{i=1}^\infty X_i)_p$.

It is well known that
\[
(X_p)^* = X^*_q = \left(\bigoplus_{i=1}^\infty X^*_i\right)_q,
\]
with the action $\ds f(x) = \sum_{i=1}^{\infty} f_i(x_i)$.

In this section, we are going to analyse the stability of the UMIP under $\ell_p$-sum for $1 <p <\infty$. \cite[Proposition 2]{WZ1} shows that the $\ell_2$-sum of spaces with UMIP may fail UMIP. Here we prove a positive result.

Let $\beta_i(t)$ be the modulus of the UMIP for $X_i$ for all $i \in \mathbb{N}$, as defined in the previous section.

\begin{theorem} \label{thm10}
Let $1 < p <\infty$. $X_p$ has the UMIP if and only if for all $i \in \N$, $X_i$ has the UMIP and $\inf_{i \in \mathbb{N}} \beta_i(t) >0$ for every $0 <t <1$.
\end{theorem}

\begin{proof}
Let $X_p$ has the UMIP. Fix $k \in \mathbb{N}$. We show that $X_k$ has the UMIP.

We use the uniformly quasicontinuous characterisation of UMIP \cite[Theorem 2.11]{BGG}.

Let $\e > 0$. Since $X_p$ has the UMIP, there exists $\del > 0$ such that for any $f \in S(X_q^*)$, there is $x \in S(X_p)$ such that if $y \in S(X_p)$ and $\|x-y\|_p < \del$, then $\|f-f_y\|_q <\e/2$ for all $f_y \in D(y)$.

Let $f_k \in S(X_k^*)$. Define $f = (f_i) \in S(X_q^*)$ by putting
\[
f_i = \left\{\begin{array}{ll} f_k & \mbox{if \, \, } i = k \\ 0 & \mbox{otherwise}\end{array} \right.
\]
By the above, there is $x \in S(X_p)$ such that if $y \in S(X_p)$ and $\|x-y\|_p < \del$, then $\|f-f_y\|_q < \e/2$ for all $f_y \in D(y)$. In particular, if $g \in D(x)$ then $\|f-g\|_q < \e/2$.

Let $z_i = x_i/\|x_i\|$ for all $i$ such that $x_i \neq 0$ and $g=(g_i)\in D(x)$.
Then, 
\beqa
1 & = & \sum_{i}^{\infty} g_i(x_i) \\
& \leq & \sum_{i}^{\infty} \|g_i\|\|x_i\| \\
  & \leq & \left(\sum_{i}^{\infty}\|g_i\|^q\right)^{1/q}\left(\sum_{i}^{\infty}\|x_i\|^p\right)^{1/p} \mbox{ (using Holder's inequality)}\\
  & = & 1.
\eeqa
So, for each $i \in \N$, $g_i(x_i) = \|g_i\|\|x_i\|$. Also, the equality case of Holder's inequality gives us that $\|x_i\|^p = \|g_i\|^q$.
Combining the two, we obtain that $\left(\frac{g_i}{\|x\|^{p/q}}\right)\left(\frac{x_i}{\|x_i\|}\right) = 1$.
Thus, $g$ has the form
\[
g_i = \left\{\begin{array}{ll} \|x_i\|^{p/q} h_i & \mbox{if \, \, } x_i \neq 0 \\ 0 & \mbox{otherwise}\end{array} \right.
\]
where $h_i \in D(z_i)$. Let $\ds \gamma_k = \sum_{i \neq k}\|x_i\|^p$. It follows that $1 = \|x\|_p^p = \|x_k\|^p + \gamma_k$. And
\[
(\e/2)^q > \sum_{i=1}^{\infty} \|f_i - g_i\|^q = \left\|f_k - \|x_k\|^{p/q} h_k\right\|^q + \gamma_k \geq \left(1 - \|x_k\|^{p/q}\right)^q + \gamma_k.
\]
It follows that $\gamma_k < (\e/2)^q$. Therefore, $\|x_k\|^p >1-(\e/2)^q$.

Let $y_k \in S(X_k)$ be such that $\|y_k-z_k\| <\del$ and $h_k \in D(y_k)$. Let $u \in S(X_p)$ be defined as
\[
u_i = \left\{\begin{array}{ll} \|x_k\|y_k & \mbox{if \, \, } i = k \\ x_i & \mbox{otherwise}\end{array}\right.
\]
We have $u \in S(X_p)$ and $\|x-u\|_p = \|x_k\| \|y_k - z_k\| <\del$. Consider $g=(g_i)$ where
\[
g_i = \left\{\begin{array}{ll} \|x_k\|^{p/q}h_k & \mbox{if \, \, } i = k \\ \|x_i\|^{p/q}f_{z_i} & \mbox{otherwise}\end{array}\right.
\]
Here $f_{z_i} \in D(z_i)$. Then $g \in D(u)$. So, as before
\[
(\e/2)^q > \|f-g\|^q_q = \left\|f_k - \|x_k\|^{p/q} h_k\right\|^q + \gamma_k.
\]
Now,
\beqa
\|f_k-h_k\| & \leq & \|f_k - \|x_k\|^{p/q}h_k\| + \left\|\|x_k\|^{p/q}h_k - h_k\right\|\\
& \leq & \e/2 + (1-\|x_k\|^{p/q})\\
& < & \e/2 + (1-(1-(\e/2)^q)^{1/q})
\leq \e/2 + \e/2 = \e.
\eeqa

Hence, $X_k$ has the UMIP and the $\delta$ that works for $\e/2$ for $X_p$ works for $\e$ for $X_k$. By \cite[Theorem 2.11]{BGG}, this implies $\inf_{i \in \mathbb{N}} \beta_i(t) >0$.

Conversely, for any $t >0$, let $\beta_0(t)=\inf_{i} \beta_i(t)$. Clearly, $\beta_0(t) >0$ for any $0 <t < 1$ and for $0 <t_1 <t_2$, $\beta_0(t_1) \leq \beta_0(t_2)$ (using Remark~\ref{rmk1}).

Let $f=(f_i) \in S(X^*_q)$. If $f_i \neq 0$ then $f_{i}/\|f_{i}\| \in S(X_{i}^*)$. So, there exists $x_{i} \in S(X_{i})$ such that
\[
h_{i} \in \{g_{i} \in B(X_{i}^*) : g_{i}(x_{i}) > 1 -\beta_0(t)\} \implies \left\|h_{i}-\frac{f_{i}}{\|f_{i}\|}\right\| < t.
\]

Define $z = (z_i) \in S(X_p)$ by
\[
z_i = \left\{\begin{array}{ll} \|f_i\|^{q/p} x_i & \mbox{if \, \, } f_i \neq 0\\ 0 & \mbox{if \, \, } f_i = 0 \end{array} \right.
\]

Let $g=(g_i) \in B(X^*_q)$ be such that $\|g-f\|_q \geq t$.
First, we deal with the case $\|g_i\| = \|f_i\|$ for all $i \in \N$.
Let $a_i = \|f_i-g_i\|$ and $b_i = \|f_i\| = \|g_i\|$. We follow the proof in \cite[Theorem 3]{Da}. Clearly, $\|f_i - g_i\| \leq \|f_i\|+\|g_i\|$, i.e., $a_i \leq 2b_i$.
Now, Let $A:=\{i \in \N : a_i/b_i >t/4\}$ and $B:= \N \setminus A$. Then,
\[
1 = \sum_{i=1}^{\infty} b_i^q \geq \sum_{i \in B} b_i^q \geq (4/t)^q \sum_{i \in B} a_i^q.
\]
That is,
\[
\sum_{i \in B} a_i^q \leq \left(t/4\right)^q.
\]

Thus,
\[
\sum_{i \in A} a_i^q = \sum_{i \in \N} a_i^q - \sum_{i \in B} a_i^q \geq t^q - (t/4)^q \geq \left(3t/4\right)^q.
\]
Hence,
\[
\alpha_0=\sum_{i \in A} b_i^q \geq (3t/8)^q.
\]

Since $\|f_i-g_i\| = a_i$ and $\|f_i\| = b_i$, we have $g_i(x_i) \leq (1-\beta_0(a_i/b_i)) b_i$. Therefore,
\beqa
g(z) & = & \sum_{i=1}^{\infty} \|f_i\|^{q/p} g_i(x_i) \leq \sum_{i=1}^{\infty} b_i^q (1-\beta_0(a_i/b_i))\\
& = & \sum_{i \in A} b_i^q(1-\beta_0(a_i/b_i))) + \sum_{i \in B} b_i^q (1-\beta_0(a_i/b_i)))\\
& \leq & (1-\beta_0(t/4)) \sum_{i \in A} b_i^q + \sum_{i \in B} b_i^q \mbox{ (follows from Remark ~\ref{rmk1})}\\
& = & (1-\beta_0(t/4))\alpha_0 + (1-\alpha_0)
= 1 - \alpha_0\beta_0(t/4)
\leq 1 -(3t/8)^q\beta_0(t/4).
\eeqa

It follows that for any $f \in S(X^*_q)$, there is $z \in S(X_p)$ with the property that, if $g \in B(X^*_q)$ is such that $\|g_i\|=\|f_i\|$ for all $i \in \N$ and $g(z) >1 - (3t/8)\beta_0(t/4)$, then $\|f-g\|_q <t$.

Now, we move to the general case.
We know that $\ell_q$ is uniformly convex. Let $\del_q(t)$ for $0 <t \leq 2$ be the modulus of convexity for $\ell_q$.

Choose $\alpha$ such that $0 <\alpha <t/2$ and $\del_q(\alpha/2)+\alpha <(3t/16)\beta_0(t/8)$.

Let $g \in B(X^*_q)$ be such that $g(z) > 1 - \del_q(\alpha/2)$.
So,
\[
\sum_{i=1}^{\infty} \|f_i\|^{q/p} \|g_i\| \geq \sum_{i=1}^{\infty} \|f_i\|^{q/p} |g_i(x_i)| > 1 - \del_q(\alpha/2).
\]

Let $F := (\|f_i\|) \in S(\ell_q)$, $G := (\|g_i\|) \in B(\ell_q)$ and $u := (\|f_i\|^{q/p}) \in S(\ell_p)$. Then,
\[
F(u) = \sum_{i=1}^{\infty} \|f_{i}\| \|f_i\|^{q/p} = 1 \quad \mbox{and} \quad G(u) > 1 - \del_q(\alpha/2).
\]
Now, by Observation ~\ref{ob2},
\[
\|F-G\|_q = \left(\sum_{i=1}^{\infty} \bigg|\|g_i\| - \|f_i\|\bigg|^q\right)^{1/q} < \alpha.
\]

Let $h=(h_i)$, where, $h_i = \frac{g_i\|f_i\|}{\|g_i\|}$. Clearly, $\|h_i\|=\|f_i\|$ for each $i$. And
\[
\|g-h\|_q = \left(\sum_{i=1}^{\infty} \bigg|\|g_i\| - \|f_i\|\bigg|^q\right)^{1/q} < \alpha.
\]

Since $g(z) >1-\del_q(\alpha/2)$, we have $h(z) \geq g(z) -\|g-h\|_q > 1-\del_q(\alpha/2) - \alpha >1 - (3t/16)\beta_0(t/8)$, by the choice of $\alpha$. By the first case,
\[
\|f-h\|_q <t/2.
\]
Hence,
\[
\|f-g\|_q \leq \|f-h\|_q + \|h-g\|_q <t.
\]
\end{proof}

\begin{corollary}
Let $1 < p <\infty$. Then $\ell_p (X)$ has the UMIP if and only if $X$ has the UMIP.
\end{corollary}

\textbf{Acknowledgements}
The author is indebted to his Ph.D. supervisor Prof. Pradipta Bandyopadhyay for the regular discussions and immense help in developing the article to its current form, Late Prof. Sudipta Dutta and Bor-Luh Lin whose work sets the background for the current work and Prof. Gilles Lancien for giving his valuable time, fruitful discussions and warm hospitality during a recent visit to Besan\c{c}on.

\bibliographystyle{amsplain}

\end{document}